\documentclass[11pt]{amsart}

\usepackage{a4wide}
\usepackage{amssymb}
\usepackage{mathrsfs}
\usepackage{enumerate, enumitem}
\usepackage{esint}
\usepackage{titletoc}
\usepackage{mathscinet}
\usepackage[colorlinks=true, urlcolor=blue, linkcolor=blue, citecolor=green]{hyperref}
\usepackage[nameinlink]{cleveref}

\theoremstyle{plain}
\newtheorem{theorem}{Theorem}[section]
\newtheorem{lemma}[theorem]{Lemma}
\newtheorem{corollary}[theorem]{Corollary}

\theoremstyle{definition}
\newtheorem{definition}[theorem]{Definition}
\newtheorem{remark}[theorem]{Remark}

\numberwithin{equation}{section}

\DeclareMathOperator*{\tr}{tr}

\makeatletter
\@namedef{subjclassname@2020}{\textup{2020} Mathematics Subject Classification}
\makeatother

\title[Non-uniformly elliptic equations in non-divergence form]{Regularity for solutions of non-uniformly elliptic equations in non-divergence form}

\author{Jongmyeong Kim}
\address{Institute of Mathematics, Academia Sinica, Taipei 106319, Republic of China}
\email{jmkim@gate.sinica.edu.tw}

\author{Se-Chan Lee}
\address{Research Institute of Mathematics, Seoul National University, Seoul 08826, Republic of Korea}
\email{dltpcks1@snu.ac.kr}

\thanks{Se-Chan Lee is supported by Basic Science Research Program through the National Research Foundation of Korea (NRF) funded by the Ministry of Education (2022R1A6A3A01086546).}

\begin{document}

\begin{abstract}
	We prove the Aleksandrov--Bakelman--Pucci estimate for non-uniformly elliptic equations in non-divergence form. Moreover, we investigate local behaviors of solutions of such equations by developing local boundedness and weak Harnack inequality. Here we impose an integrability assumption on ellipticity representing degeneracy or singularity, instead of specifying the particular structure of ellipticity.
\end{abstract}

\maketitle


\section{Introduction} \label{sec-introduction}
In this paper, we study regularity properties for solutions of non-uniformly elliptic equations in non-divergence form. To illustrate the issues, let us begin with the simplest example: a second-order, linear elliptic equation in non-divergence form 
\begin{equation}\label{eq-linear}
	a_{ij}D_{ij}u=f \quad \text{in $B_1 $},
\end{equation}
where a coefficient matrix $a=(a_{ij})_{1 \leq i, j \leq n}$ and an nonhomogeneous term $f$ are measurable. In order to capture the ellipticity of $a$, we introduce
\begin{equation}\label{eq-ellipticity}
	\lambda(x):=\inf_{\xi \in \mathbb{R}^n} \frac{\xi \cdot a(x)\xi}{|\xi|^2} \quad \text{and} \quad \Lambda(x):=\sup_{\xi \in \mathbb{R}^n} \frac{\xi \cdot a(x)\xi}{|\xi|^2}
\end{equation}
In particular, we say $a=(a_{ij})$ is \textit{uniformly elliptic} if there exist ellipticity constants $0<\lambda_0 \leq \Lambda_0<\infty$ such that
\begin{equation*}
	\lambda_0 \leq \lambda(x) \leq \Lambda(x) \leq \Lambda_0.
\end{equation*}
The regularity theory of (possibly nonlinear) uniformly elliptic operators in non-divergence form is by now classical; we refer to comprehensive books \cite{CC95, GT01} and references therein. 
In particular, Aleksandrov \cite{Ale66}, Bakelman \cite{Bak61}, and Pucci \cite{Puc66} independently proved a maximum principle: if $u \in C(\overline{B_1}) \cap W^{2, n}_{\mathrm{loc}}(B_1)$ is a strong subsolution of \eqref{eq-linear}, then there exists a constant $C=C(n, \lambda_0, \Lambda_0)>0$ such that
\begin{equation}\label{eq-abp}
		\sup_{B_1} u \leq \sup_{\partial B_1} u^++C\|f^-\|_{L^n(\Gamma^+(u^+))},
\end{equation}
where $\Gamma^+(u^+)$ is the upper contact set; see \Cref{sec-preliminaries} for the precise definition. The ABP maximum principle became a fundamental tool in establishing local estimates for the associated equations, such as local boundedness, weak Harnack inequality, and interior H\"older estimate.

The goal of this paper is to develop the ABP maximum principle and to derive interior a priori estimates for solutions of \textit{non-uniformly}, nonlinear elliptic equations. In our framework, the ellipticity functions $1/\lambda$ and $\Lambda$ are not necessarily bounded, but they satisfy some integrability conditions. To be precise, we define two measurable functions $\lambda, \Lambda : B_1 \to [0, \infty]$ such that $\lambda \leq \Lambda$,
\begin{align*}
	1/\lambda \in L^{p}(B_1), \quad \text{and} \quad \Lambda \in L^q(B_1).
\end{align*}
It is noteworthy that the uniformly elliptic case corresponds to the choice $p=q=\infty$. Moreover, we define a generalized version of \textit{Pucci extremal operators} $\mathcal{M}_{\lambda, \Lambda}^{\pm}$ by
\begin{align*}
	\mathcal{M}_{\lambda, \Lambda}^+(M)(x)&:=\Lambda(x) \sum_{e_i \geq 0} e_i(M) +\lambda(x) \sum_{e_i<0} e_i(M),\\
	\mathcal{M}_{\lambda, \Lambda}^-(M)(x)&:=\lambda(x) \sum_{e_i \geq 0} e_i(M) +\Lambda(x) \sum_{e_i<0} e_i(M), 
\end{align*}
where $x \in B_1$, $M \in \mathcal{S}^n:=\{M : \text{$M$ is an $n \times n$ real symmetric matrix}\}$, and $e_i(M)$'s are the eigenvalues of $M$. For constant ellipticity $\lambda_0$ and $\Lambda_0$, it reduces to the classical Pucci extremal operators; see \cite{CC95, CIL92} for instance. 

Throughout the paper, we assume that a pair $(p, q)$ satisfies
\begin{equation}\label{eq-assumption1}
	\frac{1}{p}+\frac{1}{q} \leq \frac{1}{n}.
\end{equation}
and we set the constants $\theta, \tau \in [n, \infty]$ to satisfy
\begin{equation}\label{assumption2}
	\frac{1}{\theta}=\frac{1}{n}-\frac{1}{p}-\frac{1}{q} \quad \text{and} \quad \frac{1}{\tau}=\frac{1}{n}-\frac{1}{p}.
\end{equation}
Then we are concerned with an $L^{\theta}$-strong solution $u$ of
\begin{equation*}
	\mathcal{M}_{\lambda, \Lambda}^+(D^2u)(x) \geq f(x)
\end{equation*}
or 
\begin{equation*}
	\mathcal{M}_{\lambda, \Lambda}^-(D^2u)(x) \leq f(x)
\end{equation*}
for an nonhomogeneous term $f \in L^{\tau}(B_1)$; see \Cref{sec-preliminaries} for details. 

We begin with the Aleksandrov-Bakelman-Pucci estimates for $L^{\theta}$-strong subsolutions. Several corollaries of \Cref{thm-ABP} are discussed at the end of \Cref{sec-ABP}.  
\begin{theorem}[ABP estimates]\label{thm-ABP}
	Let $f \in L^{\tau}(B_1)$. Suppose that $u \in W^{2, \theta}(B_1)$ is an $L^{\theta}$-strong solution of 
	\begin{align*}
		\mathcal{M}_{\lambda, \Lambda}^+(D^2u) \geq f \quad \text{in $B_1$}.
	\end{align*}
	Then there exists a universal constant $C=C(n)>0$ such that
	\begin{align*}
		\sup_{B_1} u \leq \sup_{\partial B_1} u^++C \left(\int_{\Gamma^+(u^+)} \left(\frac{f^-(x)}{\lambda(x)}\right)^n \,\mathrm{d}x\right)^{1/n}.
	\end{align*} 
\end{theorem}

After the pioneering works by Aleksandrov--Bakelman--Pucci, the ABP maximum principle has been widely studied in different contexts. Just to name a few, the ABP estimate, concerning uniformly elliptic/parabolic equations in non-divergence form, was achieved for 
\begin{enumerate}[label=(\roman*)]
	\item viscosity solutions of fully nonlinear elliptic equations \cite{Caf88, Caf89};
	
	\item strong solutions of linear parabolic equations \cite{Kry76, Tso85};
	
	\item viscosity solutions of fully nonlinear elliptic equations \cite{Wan92a};
	
	\item $L^p$-viscosity solutions of fully nonlinear elliptic/parabolic equations \cite{CCKS96, CKS00};
	
	\item viscosity solutions of fully nonlinear elliptic equations with gradient growth terms \cite{KS07, KS22}.
\end{enumerate}
We refer to \cite{Cab95, Esc93} for improvements of the ABP estimates in other directions. On the other hand, non-uniformly elliptic equations with particular structure have been considered relatively recently by several authors in various circumstances: \cite{BBLL24, DFQ09, Imb11} when an operator is given by $|Du|^{\gamma}\mathcal{M}_{\lambda_0, \Lambda_0}^{\pm}(D^2u)$ with $\gamma>-1$, \cite{ACP11} for $p$-Laplace equations and the mean curvature flow, and \cite{IS16, Moo15} for elliptic equations that hold only where the gradient is large. In this paper, we concentrate on analyzing non-uniformly elliptic equations whose degeneracy and singularity are implicitly encoded in the integrability of $1/\lambda$ and $\Lambda$.

We next move our attention to local estimates for solutions of Pucci extremal operators. We first show the local boundedness result for strong subsolutions.
\begin{theorem}[Local boundedness]\label{thm-localboundedness}
	Let $f \in L^{\tau}(B_1)$. Suppose that $u \in W^{2, \theta}_{\mathrm{loc}}(B_1)$ is an $L^{\theta}$-strong solution of 
	\begin{align*}
		\mathcal{M}_{\lambda, \Lambda}^+(D^2u) \geq f \quad \text{in $B_1$}.
	\end{align*}
	Then for $0<t \leq n$, we have
	\begin{align*}
		\sup_{B_{1/2}}u \leq C\left(\left\|(u^+)^{t/n}\Lambda/\lambda\right\|_{L^n(B_1)} ^{n/t}+\|f^-/\lambda\|_{L^n(B_1)} \right)
	\end{align*}
for a universal constant $C=C(n, t)>0$. 

In particular, for $t>0$, there exists $C=C(n, t, \|1/\lambda\|_{L^p(B_1)}, \|\Lambda\|_{L^q(B_1)})>0$ such that
\begin{align*}
	\sup_{B_{1/2}} u \leq C\left(\|u^+\|_{L^{\theta t/n}(B_1)}+\|f^-/\lambda\|_{L^n(B_1)} \right).
\end{align*}
\end{theorem}

We also prove the weak Harnack inequality for viscosity supersolutions under a stronger assumption on $(p, q)$.
\begin{theorem}[Weak Harnack inequality]\label{thm-weakharnack}
	Let $f \in L^{\tau}(B_1)$ and assume that
	\begin{align*}
		\frac{1}{p}+\frac{1}{q}<\frac{1}{2n}.
	\end{align*}
	Moreover, suppose that $u \in W^{2, \theta}_{\mathrm{loc}}(B_1)$ is an $L^{\theta}$-strong solution of 
	\begin{align*}
		\mathcal{M}_{\lambda, \Lambda}^-(D^2u) \leq f \quad \text{in $B_1$}.
	\end{align*}
	If $u$ is nonnegative in $B_{1}$, then we have
	\begin{align*}
	\left\|u\right\|_{L^t(B_{1/2})}\leq C\left(\inf_{B_{1/2}}u+\|f/\lambda\|_{L^n(B_{1})}\right)
\end{align*}
for some positive constants $t$ and $C$ which depend only on $n$, $\|1/\lambda\|_{L^p(B_1)}$, and $\|\Lambda\|_{L^q(B_1)}$.

\end{theorem}
As immediate consequences of \Cref{thm-localboundedness} and \Cref{thm-weakharnack}, we provide a Harnack inequality and an interior H\"older regularity of strong solutions in \Cref{sec-localest}.


We now describe two simple, but interesting observations regarding our main theorems:
\begin{enumerate}[label=(\roman*), wide, labelwidth=!, labelindent=0pt]
\item For $n=1$ and $\gamma>0$, let us consider a linear operator $Lu=|x|^{\gamma}u_{xx}$ in $B_1=(-1, 1)$. We then claim that $u(x)=|x|$ is a $C$-viscosity solution of $Lu=0$ in $B_1$; see \Cref{def-vis} for the definition of $C$-viscosity solutions. Indeed, for $x_0 \in B_1 \setminus \{0\}$, then $u_{xx}(x_0)=0$ and so $Lu(x_0)=0$ in the classical sense. For $x_0=0$, if we let $\varphi \in C^2(B_1)$ be a test function such that $u-\varphi$ has a local maximum (or minimum) at $0$, then 
\begin{equation*}
	L\varphi(0)=|x|^{\gamma}\varphi_{xx}|_{x=0}=0.
\end{equation*}
Therefore, we conclude that $u$ is a viscosity solution of $Lu=0$ in $B_1$. 

On the other hand, if we choose ellipticity functions $\lambda(x)=|x|^{\gamma}$ and $\Lambda(x)=1$, then it immediately follows that a viscosity solution $u$ of $Lu=0$ in $B_1$ satisfies
\begin{equation*}
	\mathcal{M}_{\lambda, \Lambda}^+(D^2u) \geq 0 \quad \text{and} \quad \mathcal{M}_{\lambda, \Lambda}^-(D^2u) \leq 0 \quad\text{in $B_1$}.
\end{equation*}
Moreover, it is easy to see that $\Lambda \in L^{\infty}(B_1)$ and $1/\lambda \in L^p(B_1)$ for any $p<1/\gamma$, while $u$ does not enjoy the (weak) minimum principle in $B_1$. Hence, even though we impose stronger integrability condition on $1/\lambda$ and $\Lambda$, \Cref{thm-ABP} does not hold for general ``viscosity solution" $u$. In other words, this example shows that the ``strong solution" condition on $u$ is essential in our framework.

\item For $n=2$, we consider a linear operator $Lu=2u_{xx}+y^2u_{yy}$ in $B_1=\{(x, y): x^2+y^2<1\}$. Then ellipticity functions are given by $\lambda(x,y)=y^2$ and $\Lambda(x,y)=2$, where $1/\lambda=|y|^{-2} \notin L^1(B_1)$. It follows from a direct calculation that $u(x, y)=y^2\cos x$ is a classical (or strong) solution of $Lu=0$ in $B_1$. Since $u(0, 0)=0=\min_{\partial B_1}u$, $u$ does not satisfy the strong maximum principle and the weak Harnack inequality. In short, this example guarantees the necessity of (a version of) integrability criterion on $1/\lambda$ and $\Lambda$ in \Cref{thm-weakharnack}. Nevertheless, the optimality of our assumption on $(p, q)$ is not verified by this example, and it remains an interesting open problem.
\end{enumerate}

Let us finally discuss similar consequences for linear non-uniformly elliptic equations in divergence form. In particular, as a variational counterpart of \eqref{eq-linear}, the authors of \cite{BS21, Tru71} considered a weak solution $u$ of
\begin{equation*}
	-D_j(a_{ij}D_iu)=0 \quad \text{in $B_1$},
\end{equation*}
where the ellipticity of $a$ is measured by $\lambda$ and $\Lambda$ defined in \eqref{eq-ellipticity}. In \cite{Tru71}, Trudinger established interior estimates such as local boundedness, Harnack inequality, and H\"older regularity for weak solutions, provided that $1/\lambda \in L^p(B_1)$, $\Lambda \in L^q(B_1)$ with
	\begin{align*}
		\frac{1}{p}+\frac{1}{q}<\frac{2}{n}.
	\end{align*} 
 Recently, Bella and Sch\"affner \cite{BS21} improved the result by replacing the condition with
	\begin{align*}
		\frac{1}{p}+\frac{1}{q}<\frac{2}{n-1},
	\end{align*}
	and proved that this integrability condition is indeed sharp. The strategy of both papers mainly relied on a modification of the Moser iteration method, which is not available for operators in non-divergence form. We also refer to \cite{CMM18, MS68} for related results.

The paper is organized as follows. In \Cref{sec-preliminaries}, we summarize several notations which will be used throughout the paper. \Cref{sec-ABP} is devoted to the proof of \Cref{thm-ABP} by adopting sequential approximation techniques. Finally, we investigate local behaviors of strong solutions: local boundedness for subsolutions and weak Harnack inequality for supersolutions.

 \section{Preliminaries} \label{sec-preliminaries}
We first introduce a concept of $L^{\theta}$-strong solutions for Pucci extremal operators $\mathcal{M}_{\lambda, \Lambda}^{\pm}$.
\begin{definition}[$L^{\theta}$-strong solutions]\label{def-strong}
	Let $f \in L^{\tau}_{\text{loc}}(B_1)$.
	A function $u \in W^{2, \theta}_{\mathrm{loc}}(B_1)$ is an $L^{\theta}$-strong solution of $\mathcal{M}_{\lambda, \Lambda}^+(D^2u) \geq f$ in $B_1$, if
	\begin{equation*}
		\mathcal{M}_{\lambda, \Lambda}^+(D^2u):=\Lambda(x) \sum_{e_i \geq 0} e_i(D^2u(x)) +\lambda(x) \sum_{e_i<0} e_i(D^2u(x)) \geq f(x) \quad \text{a.e. in $B_1$},
	\end{equation*}
	where $e_i(M)$'s are the eigenvalues of $M$.
	
	In a similar way, a function $u \in W^{2, \theta}_{\mathrm{loc}}(B_1)$ is an $L^{\theta}$-strong solution of $\mathcal{M}_{\lambda, \Lambda}^-(D^2u) \leq f$ in $B_1$, if
	\begin{equation*}
		\mathcal{M}_{\lambda, \Lambda}^-(D^2u):=\lambda(x) \sum_{e_i \geq 0} e_i(D^2u(x)) +\Lambda(x) \sum_{e_i<0} e_i(D^2u(x)) \leq f(x) \quad \text{a.e. in $B_1$}.
	\end{equation*}
\end{definition}

\begin{remark}
	 The constants $\theta$ and $\tau$ are chosen to verify that $(\Lambda/\lambda) D^2u$ and $f/\lambda$ are contained in  $L^n$-space. If $1/\lambda$ and $\Lambda$ further belong to $L^{\infty}$-space, then it corresponds to the uniformly elliptic setting with $p=q=\infty$ and $\theta=\tau=n$. In this case, \Cref{def-strong} coincides with the definition of $L^n$-strong solutions given in \cite{CCKS96}.
\end{remark}

We provide simple properties of $\mathcal{M}^{\pm}=\mathcal{M}_{\lambda, \Lambda}^{\pm}$ as follows.
\begin{lemma}\label{lem-pucci}
	Let $M, N \in \mathcal{S}^n$. Then the following hold a.e.:
	\begin{enumerate}[label=(\roman*)]
		\item $\mathcal{M}^-(M) \leq \mathcal{M}^+(M)$.
		\item $\mathcal{M}^-(M)=-\mathcal{M}^+(-M)$.
		\item $\mathcal{M}^{\pm}(t M)=t\mathcal{M}^{\pm}(M)$ if $t \geq 0$.
		\item $\mathcal{M}^+(M)+\mathcal{M}^-(N) \leq \mathcal{M}^+(M+N) \leq \mathcal{M}^+(M)+\mathcal{M}^+(N)$.    
	\end{enumerate}
\end{lemma}

For later uses, we also define $C$-viscosity solutions when $\lambda$, $\Lambda$, and $f$ are continuous; see \cite{CC95, CIL92} for instance.
\begin{definition}[$C$-viscosity solutions]\label{def-vis}
	Let $\lambda, \Lambda, f \in C(B_1)$ with $0 \leq \lambda(x) \leq \Lambda(x)$ for $x \in B_1$. A function $u \in C(B_1)$ is a $C$-viscosity solution of $\mathcal{M}_{\lambda, \Lambda}^+(D^2u) \geq f$ in $B_1$, if for all $\varphi \in C^2(B_1)$ and point $x_0 \in B_1$ at which $u-\varphi$ has a local maximum, one has 
		\begin{equation*}
		\mathcal{M}_{\lambda(x_0), \Lambda(x_0)}^+(D^2\varphi(x_0)) \geq f(x_0).
	\end{equation*}
	In a similar way, a function $u \in C(B_1)$ is a $C$-viscosity solution of $\mathcal{M}_{\lambda, \Lambda}^-(D^2u) \geq f$ in $B_1$, if for all $\varphi \in C^2(B_1)$ and point $x_0 \in B_1$ at which $u-\varphi$ has a local minimum, one has 
	\begin{equation*}
		\mathcal{M}_{\lambda(x_0), \Lambda(x_0)}^-(D^2\varphi(x_0)) \leq f(x_0).
	\end{equation*}
\end{definition}

The following contact set $\Gamma^+$ will be used for the proof of ABP estimates.
\begin{definition}
For a function $u : \Omega \to \mathbb{R}$ and $r>0$, the \textit{upper contact set} is defined by
\begin{equation*}
	\begin{aligned}
		\Gamma^+(u)&=\Gamma^+(u, \Omega)=\{x \in \Omega \text{ : $\exists\, p \in \mathbb{R}^n$ such that $u(y) \leq u(x)+\langle p, y-x\rangle$, $\forall y \in \Omega$} \},\\
		\Gamma_r^+(u)&=\Gamma_r^+(u, \Omega)=\{x \in \Omega \text{ : $\exists \, p \in \overline{B_r(0)}$ such that $u(y) \leq u(x)+\langle p, y-x\rangle$, $\forall y \in \Omega$} \}.
	\end{aligned}
\end{equation*}
\end{definition}

For sets $A_1$, $A_2$, $\cdots$, we define
\begin{align*}
	\limsup_{j \to \infty}A_j:=\bigcap_{n=1}^{\infty}\bigcup_{k>n}A_k.
\end{align*}

\begin{lemma}[{\cite[Lemma A.1]{CCKS96}}]\label{lem-contactset}
	Let $u_j$, $j=1,2, \cdots$ be functions defined on sets $\Omega_j$, where $\Omega_j$ are open and increase to $\Omega$; that is $\Omega_j \subset \Omega_{j+1}$ and $\bigcup \Omega_j=\Omega$. Let $u_j$ converge uniformly to a continuous function $u$ on each $\Omega_m$. Then
	\begin{enumerate}[label=(\roman*)]
		\item $\limsup_{j \to \infty} \Gamma^+(u_j, \Omega_j) \subset \Gamma^+(u, \Omega)$.
		\item $\limsup_{j \to \infty} |\Gamma^+(u_j, \Omega_j)| \leq |\Gamma^+(u, \Omega)|$.
		\item $\limsup_{j \to \infty} \Gamma^+_r(u_j, \Omega_j) \subset \Gamma^+_r(u, \Omega)$.
	\end{enumerate}
\end{lemma}

We finally state a version of the cube decomposition lemma, which is suitable for our purpose in \Cref{sec-localest}.

\begin{lemma}[{\cite[Lemma 9.23]{GT01}}]\label{lem-cube}
	Let $K_0$ be a cube in $\mathbb{R}^n$, $w \in L^1(K_0)$, and set
	\begin{equation*}
		D_k=\{x \in K_0 \, |\, w(x) \leq k\} \quad \text{for $k \in \mathbb{R}$}.
	\end{equation*}
	Suppose that there exist constants $\delta\in (0,1)$ and $C>0$ such that
	\begin{equation*}
		\sup_{K_0 \cap K_{3r}(z)}\,(w-k) \leq C,
	\end{equation*}
	whenever $k$ and $K=K_r(z) \subset K_0$ satisfy
	\begin{equation*}
		|D_k \cap K| \geq \delta|K|.
	\end{equation*}
	Then it follows that, for all $k$,
	\begin{equation*}
		\sup_{K_0} \, (w-k) \leq C\left(1+\frac{\log(|D_k|/|K_0|)}{\log \delta}\right).
	\end{equation*}
\end{lemma}

\subsection{Applications}
In this section, we present concrete examples of degenerate or singular equations in non-divergence form, which are contained in our framework.

\begin{enumerate}[label=(\roman*), wide, labelwidth=!, labelindent=0pt]
	\item (Issacs equations) {Issaces equations}, which naturally arise in probability theory \cite{FS89} (stochastic control and differential games), are given by	
	\begin{align*}
		\inf_{\alpha}\sup_{\beta}\left(A_{\alpha\beta}(x)D^2u(x)\right)=f \quad \text{in $B_1$},
	\end{align*}
	where $A_{\alpha \beta}(\cdot)$ (for any $\alpha$, $\beta$ contained in index sets) are matrices satisfying
	\begin{align*}
		\lambda(x) I_n \leq A_{\alpha \beta}(x) \leq \Lambda(x) I_n
	\end{align*}
	with $1/\lambda \in L^p(B_1)$ and $\Lambda \in L^q(B_1)$. We note that linear elliptic operators with ellipticity $\lambda$ and $\Lambda$, and the Pucci extremal operators $\mathcal{M}_{\lambda, \Lambda}^{\pm}$ can be understood as special cases of Issacs operators.

	\item (Monge-Amp\'ere equations)
	The {Monge-Amp\'ere} equation, which appears from the prescribed Gaussian curvature equation \cite{Fig17} (or ``Minkowski problem"), is a fully nonlinear, degenerate elliptic equation given by
	\begin{align*}
		\mathrm{det}\,D^2u=f \quad \text{in $B_1$}.
	\end{align*}
	It has important applications in convex geometry and optimal transportation. For simplicity, we consider an equation
	\begin{equation}\label{eq-mongeampere}
		G(D^2u):=\log \mathrm{det} \, D^2u=\log f.
	\end{equation}
	Then we have $G_{ij}=u^{ij}$, where $u^{ij}$ denote the inverse of the Hessian matrix $D^2u$. Thus, if we denote $\lambda$ and $\Lambda$ are ellipticity functions defined in \eqref{eq-ellipticity} for $u^{ij}$, then we observe that $1/\Lambda$ and $1/\lambda$ are ellipticity function for $D^2u$. Since
	\begin{equation*}
		\begin{aligned}
			&\text{$u$ is convex if and only if \eqref{eq-mongeampere} is degenerate elliptic and}\\
			&\text{$u$ is uniformly convex if and only if \eqref{eq-mongeampere} is uniformly elliptic},
		\end{aligned}
	\end{equation*}
	the integrability assumption on $1/\lambda$ and $\Lambda$ corresponds to some ``intermediate'' convexity on $u$.

	\item (Linear equations with particular degeneracy/singularity) In \cite{DS09}, the authors employed the partial Legendre transform to convert the two-dimensional Monge-Amp\'ere equation
	\begin{equation*}
		\mathrm{det} \, D^2u=|x|^{\alpha} \quad \text{for $\alpha>0$}
	\end{equation*}
	into the linear equation
	\begin{equation*}
		v_{xx}+|x|^{\alpha}v_{yy}=0 \quad \text{in $B_1$}.
	\end{equation*}
	Then the pair $(p,q)$ corresponding to the ellipticity functions given by $\lambda(x, y)=|x|^{\alpha}$ and $\Lambda(x, y)=1$ satisfy the stuctural condition \eqref{eq-assumption1} when $\alpha<1/2$.

	Moreover, a similar type of equation can be found in an extension problem related to the fractional Laplacian \cite{CS07}. To be precise, the solution $u$ of the degenerate/singular equations
	\begin{equation*} 
		\left\{ \begin{aligned}
			\Delta_xu+z^{\frac{2s-1}{s}}u_{zz}&=0 &&\text{in $\mathbb{R}^n \times [0, \infty)$}\\
			u&=f &&\text{on $\mathbb{R}^n \times \{0\}$}
		\end{aligned}\right.
	\end{equation*}
	satisfies 
	\begin{equation*}
		(-\Delta)^sf(x)=-C(n,s)u_z(x, 0)
	\end{equation*}
	for $s \in (0,1)$. It is easy to check that this example lies in our setting when $(n+1)/(2n+3)<s<(n+1)/(2n+1)$. We refer to \cite{KLY24} for related examples.
\end{enumerate}

\section{ABP estimates}\label{sec-ABP}
In order to prove \Cref{thm-ABP}, we are going to provide a version of \cite[Proposition 2.12]{CCKS96} (ABP estimates for continuous coefficients and $C$-viscosity solutions) and \cite[Lemma 3.1]{CCKS96} (existence of $L^n$-strong super/subsolutions). It is noteworthy that an additional approximation technique is required to control ellipticity functions $\lambda$ and $\Lambda$, which are not necessarily bounded in $L^{\infty}$.

\begin{lemma}\label{lem-ABPreduction}
	Let $f \in C(B_1)$. Assume $\lambda, \Lambda : B_1 \to (0,\infty)$  such that $1/\lambda, \Lambda \in C(\overline{B_1})$. Moreover, suppose that $u \in C(\overline{B_1})$ is a $C$-viscosity solution of 
	\begin{align*}
		\mathcal{M}_{\lambda, \Lambda}^+(D^2u) \geq f \quad \text{in $B_1$}.
	\end{align*}
	Then there exists a universal constant $C=C(n)>0$ such that
	\begin{align*}
	\sup_{B_1} u \leq \sup_{\partial B_1} u^++C \left(\int_{\Gamma^+(u^+)} \left(\frac{f^-(x)}{\lambda(x)}\right)^n \,\mathrm{d}x\right)^{1/n}.
\end{align*} 
\end{lemma}

\begin{proof}
	We will follow the proof provided in \cite[Appendix A]{CCKS96}. We begin by assuming that $u \in C^2(B_1) \cap C(\overline{B_1})$ and later remove this assumption via approximations. We set
	\begin{equation}
		r_0 = \frac{ \sup_{B_1} u - \sup _{\partial B_1} u^{+} }{2}. 
	\end{equation}
	For $r < r_0$, let $p \in B_r$ and $\hat{x} \in \overline{B_1}$ be a maximum point of $u(\cdot) - \langle p,\cdot\rangle$ so that 
	\begin{equation*}
		u(\hat{x}) - \langle p,\hat{x} \rangle \ge u(x) - \langle p,x \rangle \text{ or } u(x)-u(\hat{x}) \le \langle p,x-\hat{x} \rangle
	\end{equation*}
	for any $x \in \overline{B_1} $. It follows that 
	\begin{align*}
		\sup_{B_1} u-u(\hat{x})  \le 2|p| \le 2r < 2r_0   = \sup _{B_1} u - \sup_{ \partial B_1 } u^+
	\end{align*}   
	and then $2(r_0- r) + \sup_{\partial B_1} u^+ < u(\hat{x}). $ In particular, $\hat{x} \in B_1 $ and $ u ( \hat{x}) >0 $. Since $ Du(\hat{x})=p $ and $D^2u (\hat{x}) \le 0$, we conclude that for $0 <r < r_0$, $\Gamma^+_r(u^+)$ is a compact subset of $B_1$ and
	\begin{equation}
		\label{eq:A.4}
		B_r=B_r(0) = Du( \Gamma^+_r(u^+) ) \text{ and } D^2u(x) \ge 0\text{ on } \Gamma^+_r(u^+) \subset \{ u > 0 \} .
	\end{equation}
	We now employ the change of variables $p = Du(x)$ to have
	\begin{equation}\label{eq:A.5} 
		\int_{B_r}  \,\mathrm{d} p \le \int_{\Gamma^+_r(u^+) }    |\det D^2u| \,\mathrm{d} x 
		\le \int_{ \Gamma^+_r(u^+) }    \left(  \frac{ -\tr D^2u }{n} \right) ^n \,\mathrm{d} x . 
	\end{equation}
	Since $ \mathcal{M}_{\lambda, \Lambda}^+(D^2u)(x) \geq f(x)  $ and $D^2u \le 0 $ on $\Gamma^+(u^+)$, we have
	\begin{equation*}
		\lambda \tr(D^2u) \ge f(x) \text{ on } \Gamma^+(u^+) 
	\end{equation*}
	and \eqref{eq:A.5} implies
	\begin{equation*} 
		r^n|B_1|=\int_{B_r}    \,\mathrm{d} p  \le \frac{1}{n^n}  \int_{ \Gamma^+(u^+)}     \left(  \frac{f^-(x)}{\lambda(x)} \right) ^n \,\mathrm{d} x.
	\end{equation*}
	Since $\lambda$, $\Lambda$, and $f$ are continuous, the general case follows from the standard approximation argument as in \cite[Appendix A]{CCKS96}.
\end{proof}

\begin{lemma}\label{lem-existence}
	Let $f\in L^{n}(B_1)$, $\psi \in C(\partial B_1)$. Assume $1/\lambda, \Lambda \in C(\overline{B_1})$. Then there exist $L^n$-strong solutions $u, v \in C(\overline{B_1}) \cap W^{2, n}_{\mathrm{loc}}(B_1)$ of 
	\begin{align*}
		\mathcal{M}^+_{\lambda, \Lambda}(D^2u)(x) \leq f(x) \quad \text{in $B_1$}
	\end{align*}
and
	\begin{align*}
	\mathcal{M}^-_{\lambda, \Lambda}(D^2v)(x) \geq f(x)\quad \text{in $B_1$}
\end{align*}
 such that $u=v=\psi$ on $B_1$. Moreover, $u$, $v$ satisfy the estimate
\begin{align}\label{eq-abplike}
	\|u\|_{L^{\infty}(B_1)}, \|v\|_{L^{\infty}(B_1)} \leq \|\psi\|_{L^{\infty}(\partial B_1)}+C\|f/\lambda\|_{L^n(B_1)}.
\end{align}
\end{lemma}

\begin{proof}
	The existence of a strong solution follows from \cite[Lemma 3.1]{CCKS96}. The $L^{\infty}$-estimate is a consequence of \Cref{lem-ABPreduction} together with the approximation.
\end{proof}

We are now ready to prove the first main theorem [\Cref{thm-ABP}].
\begin{proof}[Proof of \Cref{thm-ABP}]
	We employ several regularization techniques; more precisely, we approximate the ellipticity $\lambda$, $\Lambda$, and then the forcing term $f$. For simplicity, we may omit ``a.e." if no confusion occurs.
	\begin{enumerate}[label=(\roman*), wide, labelwidth=!, labelindent=0pt]
		\item (Approximation of $\lambda$, $\Lambda$) We first let 
		\begin{align*}
			\widetilde{\lambda}_j:=\left(\lambda \lor j \right)\land j^{-1} \quad \text{and} \quad \widetilde{\Lambda}_j:=\left(\Lambda \lor j \right)\land j^{-1}
		\end{align*}
		so that
		\begin{align*}
			\|1/\lambda_j-1/\widetilde{\lambda}_j\|_p \to 0, \quad \|\Lambda_j-\widetilde{\Lambda}_j\|_q \to 0, \quad \text{and} \quad
			j^{-1} \leq \widetilde{\lambda}_j \leq \widetilde{\Lambda}_j \leq j.
		\end{align*}
		Since $C_c(B_1)$ is a dense subspace of $L^p(B_1)$ for any $p \in [1, \infty)$, there exist two sequences $\{\lambda_j\}_{j=1}^{\infty} \subset C_c(B_1)$ and $\{\Lambda_j\}_{j=1}^{\infty} \subset C_c(B_1)$ such that
		\begin{align*}
			\|1/\lambda_j-1/\widetilde{\lambda}_j\|_{p} <j^{-1} \quad \text{and} \quad \|\Lambda_j-\widetilde{\Lambda}_j\|_{q} <j^{-1}.
		\end{align*}
		Here we may assume $\lambda_j \in [j^{-1}, j]$ and $\lambda_j \leq \widetilde{\lambda}_j$ by replacing $\lambda_j$ with $\left(\lambda_j \lor j \right)\land j^{-1}$ and $\lambda_j \land \widetilde{\lambda}_j$, if necessary. A similar argument also holds for $\Lambda_j$. Hence, it follows that $\lambda_j, \Lambda_j \in C(\overline{B_1})$, $j^{-1} \leq \lambda_j \leq \Lambda_j \leq j$,
		\begin{align}\label{eq-limit0}
			\|1/\lambda_j-1/\lambda\|_p \to 0, \quad \text{and} \quad \|\Lambda_j-\Lambda\|_q \to 0.
		\end{align}
		We now would like to find the inequality satisfied by $u$, in terms of Pucci extremal operators with `good' ellipticity $\lambda_j$ and $\Lambda_j$. Indeed, since  $u \in W^{2, \theta}(B_1)$ satisfies
		\begin{align*}
			\mathcal{M}_{\lambda, \Lambda}^+(D^2u)=\Lambda(x) \sum_{e_i>0}e_i(D^2u(x))+\lambda(x)\sum_{e_i<0}e_i(D^2u(x)) \geq f(x),
		\end{align*}
		we observe that
		\begin{align*}
			\mathcal{M}_{\lambda_{j}, \Lambda_{j}}^+(D^2u) &= \Lambda_{j}\sum_{e_i>0}e_i(D^2u)+\lambda_{j}\sum_{e_i<0}e_i(D^2u) \\
			&=\Lambda \sum_{e_i>0}e_i(D^2u)+(\Lambda_{j}-\Lambda) \sum_{e_i>0}e_i(D^2u)\\
			& +\lambda\sum_{e_i<0}e_i(D^2u)+(\lambda_{j}-\lambda)\sum_{e_i<0}e_i(D^2u)=:f_j.
		\end{align*}
		By recalling that $f \in L^{\tau}(B_1)$, $\Lambda \in L^q(B_1)$, and $D^2u \in L^{\theta}(B_1)$, it turns out that
		\begin{equation*}
			f_j=f+(\Lambda_{j}-\Lambda) \sum_{e_i>0}e_i(D^2u)+(\lambda_{j}-\lambda)\sum_{e_i<0}e_i(D^2u) \in L^{n}(B_1).
		\end{equation*}

		\item (Approximation of $f_j$) For fixed $j \in \mathbb{N}$, let $\{f_{j, k}\}_{k=1}^{\infty} \subset C^{\infty}(B_1)$ be a sequence of smooth functions such that 
		\begin{align}\label{eq-limit1}
			\|f_{j, k}-f_j\|_{L^n(B_1)} \to 0 \quad \text{as $k \to \infty$}.
		\end{align}
		Then we let $\psi_{j, k} \in W^{2, n}_{\text{loc}}(B_1) \cap C(\overline{B_1})$ solve
	\begin{equation*} 
		\left\{ \begin{aligned}
			\mathcal{M}_{\lambda_j, \Lambda_j}^-(D^2\psi_{j, k})& \geq f_{j, k}-f_{j} &&\text{in $B_1$}\\
			\psi_{j, k}&=0 &&\text{on $\partial B_1$},
		\end{aligned}\right.
	\end{equation*}
whose existence is guaranteed by \Cref{lem-existence}. From the estimate \eqref{eq-abplike},
\begin{align*}
		\|\psi_{j, k}\|_{L^{\infty}(B_1)} \leq C\|(f_{j, k}-f_j)/\lambda_j\|_{L^n(B_1)},
\end{align*}
 where the constant $C>0$ is independent of $k\in \mathbb{N}$. Therefore, it immediately follows that
\begin{align}\label{eq-limit2}
	\|\psi_{j, k}\|_{L^{\infty}(B_1)} \to 0 \quad \text{as $k \to \infty$}.
\end{align}

\item (Conclusion; ABP estimates)
	If we set $w:=u+\psi_{j, k}-\|\psi_{j, k}\|_{L^{\infty}(B_1)}$, then we observe that
	\begin{align*}
		\mathcal{M}_{\lambda_{j}, \Lambda_{j}}^+(D^2w) &\geq  \mathcal{M}_{\lambda_{j}, \Lambda_{j}}^+(D^2u)+\mathcal{M}_{\lambda_{j}, \Lambda_{j}}^-(D^2\psi_{j, k})\\
		&\geq f_j+(f_{j, k}-f_j)=f_{j, k}.
	\end{align*}
	Since $\lambda_j$, $\Lambda_j$, $f_{j, k}$ are regularized enough so that \Cref{lem-ABPreduction} is applicable, we have
	\begin{align*}
			\sup_{B_1} w \leq \sup_{\partial B_1} w^++C \left(\int_{\Gamma^+(w^+)} \left(\frac{f_{j, k}^-(x)}{\lambda_j(x)}\right)^n \, \mathrm{d} x\right)^{1/n}.
	\end{align*}
	By letting $k \to \infty$ together with \eqref{eq-limit1}, \eqref{eq-limit2}, and \Cref{lem-contactset}, we deduce
	\begin{align*}
		\sup_{B_1} u \leq \sup_{\partial B_1} u^++C \left(\int_{ \Gamma^+(u^+)} \left(\frac{f_{j}^-(x)}{\lambda_j(x)}\right)^n \, \mathrm{d} x\right)^{1/n}.
	\end{align*}
	Moreover, by applying H\"older inequality, we obtain
		\begin{align*}
			\left\|\frac{f_j^-}{\lambda_j}-\frac{f^-}{\lambda}\right\|_n&\leq
			\left\|\frac{f_j}{\lambda_j}-\frac{f}{\lambda}\right\|_n\\
			&\leq \left\|\left(\frac{1}{\lambda}-\frac{1}{\lambda_j}\right)\lambda \sum_{e_i<0}e_i(D^2u)\right\|_n+\left\|\frac{1}{\lambda_j}(\Lambda_{j}-\Lambda)\sum_{e_i>0}e_i(D^2u)\right\|_n+\left\|\frac{f}{\lambda_j}-\frac{f}{\lambda} \right\|_n \\
			&\leq \|1/\lambda-1/\lambda_j\|_p\|\Lambda\|_q\|D^2u\|_{\theta}+\|1/\lambda_j\|_p\|\Lambda_j-\Lambda\|_q\|D^2u\|_{\theta}+\|1/\lambda_j-1/\lambda\|_p\|f\|_{\tau}.
		\end{align*}
	Therefore, by passing the limit $j \to \infty$ together with \eqref{eq-limit0}, we finally conclude that
	\begin{align*}
	\sup_{B_1} u \leq \sup_{\partial B_1} u^++C \left(\int_{\Gamma^+(u^+)} \left(\frac{f^-(x)}{\lambda(x)}\right)^n \, \mathrm{d} x\right)^{1/n}
	\end{align*}
	as desired.
	\end{enumerate}
\end{proof}

\begin{corollary}\label{rmk-abp}
	Let $f \in L^{\tau}(B_1)$. Suppose that $u \in W^{2, \theta}(B_1)$ is an $L^{\theta}$-strong solution of 
	\begin{align*}
		\mathcal{M}_{\lambda, \Lambda}^-(D^2u) \leq f \quad \text{in $B_1$}.
	\end{align*}
	Then there exists a universal constant $C=C(n)>0$ such that
	\begin{align*}
		\sup_{B_1} u^- \leq \sup_{\partial B_1} u^-+C \left(\int_{\Gamma^+(u^-)} \left(\frac{f^+(x)}{\lambda(x)}\right)^n \,\mathrm{d}x\right)^{1/n}.
	\end{align*} 
\end{corollary}

\begin{proof}
	It immediately follows by considering $-u$ instead of $u$ in the proof of \Cref{thm-ABP}.
\end{proof}

\begin{corollary}
	Let $\lambda, \Lambda, f \in C(B_1)$ with $0 \leq \lambda(x) \leq \Lambda(x)$ for $x \in B_1$. Suppose that $u \in W^{2, \theta}_{\mathrm{loc}}(B_1)$ is an $L^{\theta}$-strong solution of $\mathcal{M}_{\lambda, \Lambda}^+(D^2u) \geq f$ in $B_1$. Then $u$ is also a $C$-viscosity solution of $\mathcal{M}_{\lambda, \Lambda}^+(D^2u) \geq f$ in $B_1$.
\end{corollary}

\begin{proof}
	Since $\theta \geq n$, we have $u \in C(B_1)$. We assume by contradiction that for some $\varphi \in C^2(B_1)$, $u-\varphi$ has a (strict) local maximum at $x_0 \in B_1$ and 
	\begin{equation*}
		\mathcal{M}_{\lambda(x_0), \Lambda(x_0)}^+(D^2\varphi(x_0)) < f(x_0).
	\end{equation*}
	By the continuity of $\lambda, \Lambda, f$, we have
	\begin{equation*}
		\mathcal{M}_{\lambda, \Lambda}^+(D^2\varphi) < f
	\end{equation*}
	near $x_0$. On the other hand, we observe from \Cref{lem-pucci} that
	\begin{equation*}
		\mathcal{M}_{\lambda, \Lambda}^+(D^2(u-\varphi)) \geq \mathcal{M}_{\lambda, \Lambda}^+(D^2u)-\mathcal{M}_{\lambda, \Lambda}^+(D^2\varphi)>0 \quad \text{a.e. in $B_{\eta}(x_0)$ for some $\eta>0$}.
	\end{equation*}
We now apply \Cref{thm-ABP} in $B_{\eta}(x_0)$ to conclude that
\begin{equation*}
	(u-\varphi)(x_0) \leq \sup_{\partial B_{\eta}(x_0)} (u-\varphi),
\end{equation*}
which leads to a contradiction.
\end{proof}

We say a measurable function $F: \mathcal{S}^n \times B_1 \to \mathbb{R}$ is $(\lambda(\cdot), \Lambda(\cdot))$-elliptic if
\begin{equation*}
	\mathcal{M}_{\lambda, \Lambda}^-(N)(x) \leq F(M+N, x)-F(M, x) \leq \mathcal{M}_{\lambda, \Lambda}^+(N)(x)
\end{equation*}
for any $M, N \in \mathcal{S}^n$ and $x \in B_1$ a.e.. We note that the Pucci extremal operators $\mathcal{M}_{\lambda, \Lambda}^{\pm}$ are $(\lambda(\cdot), \Lambda(\cdot))$-elliptic. The notion of $L^{\theta}$-strong solution defined in \Cref{def-strong} can be easily extended to such fully nonlinear operators $F$.

\begin{corollary}[Comparison principle]
	Let $f \in L^{\tau}(B_1)$ and $F$ be $(\lambda(\cdot), \Lambda(\cdot))$-elliptic. Suppose that $u,v \in W^{2, \theta}(B_1)$ are $L^{\theta}$-strong subsolution and supersolution of $F(D^2w,x)=f(x)$ in $B_1$, respectively. If $u \leq v$ on $\partial B_1$, then $u \leq v$ in $B_1$.		
\end{corollary}	

\begin{proof}
	By the definition of $(\lambda(\cdot), \Lambda(\cdot))$-ellipticity, we have
	\begin{equation*}
		\mathcal{M}_{\lambda, \Lambda}^+(D^2(u-v))(x) \geq F(D^2u, x)-F(D^2v, x)  \geq 0.
	\end{equation*}
	The desired result follows from \Cref{thm-ABP}.
\end{proof}

\section{Local estimates}\label{sec-localest}
In this section, we develop interior a priori estimates of $L^{\theta}$-strong solutions of non-uniformly elliptic Pucci extremal operators. The proof utilizes the ABP maximum principle [\Cref{thm-ABP}].

We begin with the local boundedness for $L^{\theta}$-strong subsolutions.
\begin{proof}[Proof of \Cref{thm-localboundedness}]
	For simplicity, we omit ``a.e." if no confusion occurs. For $\beta \geq 2$ to be determined, we define a cut-off function $\eta$ by
	\begin{align}\label{eq-cutoff}
		\eta(x)=(1-|x|^2)^{\beta}.
	\end{align}
	Then we immediately observe
	\begin{align*}
		D_i\eta&=-2\beta x_i (1-|x|^2)^{\beta-1}.\\
		D_{ij}\eta&=-2\beta\delta_{ij}(1-|x|^2)^{\beta-1}+4\beta(\beta-1)x_ix_j(1-|x|^2)^{\beta-2}.
	\end{align*}
	By setting $v=\eta u$, we have
	\begin{align*}
		\mathcal{M}_{\lambda, \Lambda}^+(D^2v)&=\mathcal{M}_{\lambda, \Lambda}^+(\eta D^2u+Du\otimes D\eta+D\eta\otimes Du+u D^2\eta) \\
		&\geq \mathcal{M}_{\lambda, \Lambda}^+(\eta D^2u)+\mathcal{M}_{\lambda, \Lambda}^-(Du\otimes D\eta+D\eta\otimes Du+u D^2\eta)\\
		&=:I_1+I_2,
	\end{align*}
	where we write $(x \otimes y)_{ij}=x_iy_j$ for $x, y \in \mathbb{R}^n$. For $I_1$, we have
		\begin{align*}
			I_1=\eta \mathcal{M}_{\lambda, \Lambda}^+(D^2u) \geq \eta f \geq -f^-.
		\end{align*}
	For $I_2$, we let $\Gamma^+_v$ be the upper contact set of $v^+$ in $B_1$; then we have $u \geq 0$ and $v$ is concave on $\Gamma_v^+$. Moreover, on $\Gamma_v^+$, we observe that
		\begin{align*}
			|Du|&=\frac{1}{\eta}|Dv-uD\eta| \leq\frac{1}{\eta}\left(\frac{v}{1-|x|}+u|D\eta|\right)\leq 2(1+\beta)\eta^{-1/\beta}u.
		\end{align*}
	It follows from the estimates
	\begin{align*}
		|Du||D\eta| &\leq 4\beta(1+\beta) \eta^{-2/\beta}v\leq 8\beta^2 \eta^{-2/\beta}v,\\
		u|D^2\eta| &\leq (2\beta \eta^{1/\beta}+4\beta(\beta-1))\eta^{-2/\beta}v\leq 4\beta^2 \eta^{-2/\beta}v,
	\end{align*}
	that
	\begin{align*}
		I_2 \geq -20\Lambda \beta^2 \eta^{-2/\beta}v.
	\end{align*}
	We now apply the ABP estimates [\Cref{rmk-abp}] to derive 
		\begin{align*}
		\sup_{B_1} v &\leq C \left(\int_{\Gamma_v^+}\left[ \left(    \frac{\beta^2\Lambda(x) \eta^{-2/\beta}(x)v^+(x)}{\lambda(x)}\right)^n+\left(\frac{f^-(x)}{\lambda(x)}\right)^n\right] \,\mathrm{d}x\right)^{1/n}\\
		&\leq C\left( \left(\sup_{B_1} v^+\right)^{1-2/\beta} \left\|(u^+)^{2/\beta}\Lambda/\lambda\right\|_{L^n(B_1)}   +\|f^-/\lambda\|_{L^n(B_1)}\right)
	\end{align*} 
	By choosing $\beta=2n/t$ (provided that $t \leq n$) and using Young's inequality,
	\begin{align*}
		\left(\sup_{B_1} v^+\right)^{1-t/n} \leq \varepsilon \sup_{B_1} v^++\varepsilon^{1-n/t} \quad \text{for any $\varepsilon>0$}.
	\end{align*}
	In particular, the choice
	\begin{align*}
		\varepsilon=\frac{1}{2C}\left\|(u^+)^{t/n}\Lambda/\lambda\right\|_{L^n(B_1)}  ^{-1}
	\end{align*}
	yields that
	\begin{align*}
		\sup_{B_{1/2}} u \leq C\left(\left\|(u^+)^{t/n}\Lambda/\lambda\right\|_{L^n(B_1)} ^{n/t}+\|f^-/\lambda\|_{L^n(B_1)} \right).
	\end{align*}
Finally, an application of H\"older's inequality concludes that
\begin{align*}
	\sup_{B_{1/2}} u \leq C\left(\|1/\lambda\|_p^{n/t}\|\Lambda\|_{q}^{n/t}\|u^+\|_{\theta t/n}+\|f^-/\lambda\|_{n} \right).
\end{align*}
\end{proof}

We now move our attention to the weak Harnack inequality for $L^{\theta}$-strong supersolutions.
\begin{proof}[Proof of \Cref{thm-weakharnack}]
	For $\varepsilon>0$, we set
	\begin{align*}
		\overline{u}&=u+\varepsilon+\|f/\lambda\|_{L^n(B_{1})},\\
		w&=-\log \overline{u}, \quad v=\eta w, \quad \text{and} \quad g=f/\overline{u},
	\end{align*}
	where $\eta$ is the cut-off function defined by \eqref{eq-cutoff} with $\beta \geq 2$ to be determined. It is easily checked that
	\begin{align*}
		D_iw&=-\overline{u}^{-1}D_i\overline{u},\\
		D_{ij}w&=\overline{u}^{-2}D_i\overline{u}D_j\overline{u}-\overline{u}^{-1}D_{ij}\overline{u}=D_iwD_jw-\overline{u}^{-1}D_{ij}{u}.
	\end{align*}
	Then a direct calculation yields that
	\begin{equation}\label{eq-comp1}
	\begin{aligned}
	&\mathcal{M}_{\lambda, \Lambda}^+(D^2v)\\
	&=\mathcal{M}_{\lambda, \Lambda}^+(\eta D^2w+2Dw\otimes D\eta+w D^2\eta) \\
	&=\mathcal{M}_{\lambda, \Lambda}^+(-\eta \overline{u}^{-1}D^2{u}+\eta Dw \otimes Dw+Dw\otimes D\eta+D\eta\otimes Dw+w D^2\eta)\\
	&\geq \mathcal{M}_{\lambda, \Lambda}^+(-\eta \overline{u}^{-1}D^2{u})+\mathcal{M}_{\lambda, \Lambda}^-(\eta Dw \otimes Dw+Dw\otimes D\eta+D\eta\otimes Dw)+\mathcal{M}_{\lambda, \Lambda}^-(w D^2\eta)\\
	&\geq -g\eta+\mathcal{M}_{\lambda, \Lambda}^-(\eta Dw \otimes Dw+Dw\otimes D\eta+D\eta\otimes Dw)+\mathcal{M}_{\lambda, \Lambda}^-(w D^2\eta).
\end{aligned}
\end{equation}

\begin{enumerate}[label=(\roman*)]
	\item We first prove the following Cauchy-Schwarz inequality for matrices:
	\begin{align*}
		\pm(Dw \otimes D\eta +D\eta \otimes Dw) \leq \eta^{-1}D\eta\otimes D\eta+\eta Dw\otimes Dw.
	\end{align*}
	It can be written in an equivalent form:
	\begin{align*}
		\left|\langle (Dw \otimes D\eta +D\eta \otimes Dw)a, a \rangle \right| \leq \langle (\eta^{-1}D\eta\otimes D\eta+\eta Dw\otimes Dw)a, a \rangle \quad \text{for any $a \in \mathbb{R}^n$}.
	\end{align*}
	Indeed, this inequality follows from a simple observation
	\begin{align*}
		\langle (b \otimes c)a, a\rangle=[(b\otimes c)a]_i a_i=(b\otimes c)_{ij}a_ja_i=a_ib_i a_jc_j=\langle a, b \rangle \langle a, c\rangle
	\end{align*}
	for any $a, b, c \in \mathbb{R}^n$.
	
	\item We control the term $\eta^{-1}|D\eta|^2$ as 
	\begin{align*}
		\eta^{-1}|D\eta|^2 \leq 4\beta^2\eta^{1-2/\beta}.
	\end{align*}
	
	\item The eigenvalues of $D^2\eta$ are 
	\begin{align*}
		4\beta(\beta-1)(1-|x|^2)^{\beta-2}|x|^2-2\beta(1-|x|^2)^{\beta-1} \quad &\text{with multiplicity $1$},\\
		-2\beta(1-|x|^2)^{\beta-1} \quad &\text{with multiplicity $n-1$}.
	\end{align*}
	We note that the first eigenvalue is nonnegative if $\alpha \leq |x| \leq 1$ and $\beta \geq 1+1/(2\alpha^2)$ for $\alpha:=1/(3n)$.
	Therefore, for $\alpha \leq |x| \leq 1$ and $\beta \geq 1+1/(2\alpha^2)$, we obtain 
	\begin{equation*}
		\begin{aligned}
			&\mathcal{M}_{\lambda, \Lambda}^-(D^2\eta)\\
			 &= \lambda [4\beta (\beta-1)(1-|x|^2)^{\beta-2}|x|^2-2\beta (1-|x|^2)^{\beta-1}] -\Lambda(n-1)[2\beta (1-|x|^2)^{\beta-1}]\\
			&= \lambda [4\beta (\beta-1)(1-|x|^2)^{\beta-2}|x|^2] -(\lambda+(n-1)\Lambda)[2\beta (1-|x|^2)^{\beta-1}]\\
			&=2\beta (1-|x|^2)^{\beta-2}\left[2\lambda(\beta-1)|x|^2-(\lambda+(n-1)\Lambda) (1-|x|^2)\right].
		\end{aligned}
	\end{equation*}
	On the other hand, if $|x| \leq \alpha$, then
		\begin{equation*}
		\begin{aligned}
			\mathcal{M}_{\lambda, \Lambda}^-(D^2\eta) \geq -\Lambda n[2\beta (1-|x|^2)^{\beta-1}].
		\end{aligned}
	\end{equation*}
\end{enumerate}
By taking the previous estimates obtained in (i), (ii), and (iii) account into \eqref{eq-comp1}, we have
\begin{align*}
		\mathcal{M}_{\lambda, \Lambda}^+(D^2v) &\geq -g\eta-\eta^{-1}\mathcal{M}_{\lambda, \Lambda}^+(D\eta \otimes D\eta)+\mathcal{M}_{\lambda, \Lambda}^-(w D^2\eta)\\
		&=-g\eta-\eta^{-1}\Lambda|D\eta|^2+\mathcal{M}_{\lambda, \Lambda}^-(w D^2\eta)\\
		&\geq -|g|-4\beta^2\Lambda-\frac{2\Lambda n\beta}{ 1-\alpha^2} v^+\mathbf{1}_{\{|x| \leq \alpha\}}\\
		&\quad+2\beta (1-|x|^2)^{-2}\left[2\lambda(\beta-1)|x|^2-(\lambda+(n-1)\Lambda) (1-|x|^2)\right]v^+ \mathbf{1}_{\{|x| \geq \alpha\}}\\
		&=:\widetilde{f}
\end{align*}
on $\Gamma_v^+$. We now apply the ABP estimates [\Cref{rmk-abp}] to derive 
\begin{align*}
	\sup_{B_1} v &\leq C \left(\int_{\Gamma_v^+}\left(\frac{\widetilde{f}^-(x)}{\lambda(x)}\right)^n \,\mathrm{d}x\right)^{1/n}.
\end{align*} 
Therefore, by recalling that $\|g/\lambda\|_{n}\leq 1$, we obtain
\begin{align*}
	\sup_{B_1} v &\leq C+C\beta^2\|\Lambda/\lambda\|_n+C\beta \|\Lambda/\lambda\|_n \sup_{B_1} v \cdot |\{|x| \leq \alpha\} \cap \{v>0\}|^{1/n} \\
	&\quad+C\sup_{B_1} v \left(\int_{\alpha\leq |x|\leq 1}\left[ \left( \frac{\Lambda}{\lambda}-\frac{\beta}{1-|x|^2} \right)_+ \frac{\beta}{1-|x|^2}\right]^n  \,\mathrm{d}x \right)^{1/n}.
\end{align*}
Since $1/p+1/q<1/(2n)$, an application of H\"older inequality yields that
	\begin{align*}
    \int_{\alpha\leq |x|\leq 1}\left[ \left( \frac{\Lambda}{\lambda}-\frac{\beta}{1-|x|^2} \right)_+ \frac{\beta}{1-|x|^2}\right]^n \,\mathrm{d}x & \leq  \int_{\{\alpha\leq |x|\leq 1\} \cap U_{\beta} } \left(  \frac{\Lambda}{\lambda}  \right) ^{2n }  \,\mathrm{d} x
    \\ & \leq \|1/\lambda\|_p^{2n}\|\Lambda\|_q^{2n} \left|U_{\beta}\right|^{1-(2n)/p-(2n)/q}
 \end{align*}
where
\begin{equation*}
	U_{\beta}:=\left\{|x| \leq 1: \frac{\Lambda(x)}{\lambda(x)} \geq \frac{\beta}{1-|x|^2}\right\}.
\end{equation*}
 We also have the following inequality:
   \begin{align*}
     |U_{\beta}| \leq \left| \left\{\frac{\Lambda}{\lambda} \geq \beta\right\}\right| \leq \beta^{-2n} \int |\Lambda/\lambda|^{2n}.
   \end{align*}
	Hence, there exists a constant $\beta>0$ which depends only on $ \left \| 1/\lambda \right \|_{p}$, $\| \Lambda \|_{L^q}$, and $n$ such that
	\begin{align*}
	\int_{\alpha\leq |x|\leq 1}\left[ \left( \frac{\Lambda}{\lambda}-\frac{\beta}{1-|x|^2} \right)_+ \frac{\beta}{1-|x|^2}\right]^n  \,\mathrm{d}x \leq \frac{1}{(2C)^n}.
	\end{align*}
  By combining all estimates above, we conclude that
   \begin{align}  \label{eq:ABP_nonuniformly:2}
     \sup_{B_1} v \le C+C\sup_{B_1} v \cdot |\{|x| \leq \alpha\} \cap \{v>0\}|^{1/n} .
   \end{align}
	In order to finish the proof, we would like to exploit the cube decomposition lemma [\Cref{lem-cube}]. For this purpose, let us define $K_R(z)$ be the open cube, parallel to the coordinate axes, with center $z$ and the side length $2R$. Since $B_{\alpha}  \subset K _{\alpha}(0) \subset \subset B_1$ for $\alpha=1/(3n)$, we have
    \begin{align*}
   \sup_{B_1} v \le  C(1+ \sup_{B_1} v^+ |K_{\alpha}^+|^{\frac{1}{n}}),
   \end{align*}
   where $K_{\alpha}^+:= \{x \in K_{\alpha} \,| \,v>0\}$. Hence, whenever
    \begin{align*}
      \frac{|K_{\alpha}^+|}{ |K_{\alpha}|} \le \theta := [2(2\alpha)^nC]^{-1}, 
   \end{align*}
we obtain
 \begin{align*}
\sup_{B_1} v \le 2C.
\end{align*}
We point out that 
\begin{enumerate}[label=(\roman*)]
	\item this procedure is stable under the transformation $x \to \alpha(x-z)/r$ for $B_{r/\alpha}(z) \subset B_1(0)$;
	
	\item we can repeat this argument for $w-k$ instead of $w$ for arbitrary $k \in \mathbb{R}$.
\end{enumerate}
Thus, by applying \Cref{lem-cube} with $\delta = 1-\theta, \; K_0 = K_{\alpha}(0)$, and $\alpha=1/(3n)$, we obtain 
 \begin{align*}
\sup_{K_0} \, (w-k) \le  C \left( 1 + \frac{ \log (|D_k| / |K_0|)}{\log \delta} \right),
\end{align*}
where $D_k= \left\{ x \in K_0 \,|\, w(x) \le k \right\}$. In other words, if we write  
 \begin{align*}
\mu_t = |\{  x \in K_0\,| \,\bar{u} >t \}| \quad \text{with $t=e^{-k}$},
\end{align*}
 then
 \begin{align*}
\mu_t \le C\left(\inf_{K_0} \bar{u}/t\right)^{\kappa},
\end{align*} 
where $C$ and $\kappa$ are positive universal constants. By recalling \cite[Lemma 9.7]{GT01}, we obtain 
 \begin{align*}
\int_{B_{\alpha}} \bar{u}^t \le C \left(\inf_{B_{\alpha}}\bar{u}\right)^t,
\end{align*}
for $t=\kappa/2$. The desired weak Harnack inequality follows by letting $\varepsilon \to 0$ together with the covering and scaling argument. 
\end{proof}

We remark that if $u$ is a strong solution of $F(D^2u, x)=f(x)$ for a $(\lambda(\cdot), \Lambda(\cdot))$-elliptic operator $F$ with $F(0, x)=0$,  then $u$ is contained in the \textit{(extended) Pucci class}, i.e., $u$ satisfies
\begin{equation*}
		\mathcal{M}_{\lambda, \Lambda}^+(D^2u) \geq -|f| \quad \text{and} \quad 	\mathcal{M}_{\lambda, \Lambda}^-(D^2u) \leq |f|.
\end{equation*}
Indeed, the following corollaries hold for a wide class of functions: not only solutions of degenerate/singular fully nonlinear equations, but also functions in the (extended) Pucci class.

\begin{corollary}[Harnack inequality]\label{cor-harnack}
	Let $f \in L^{\tau}(B_1)$. Assume that
	\begin{align*}
		\frac{1}{p}+\frac{1}{q}<\frac{1}{2n}.
	\end{align*}
	Moreover, suppose that $u \in W^{2, \theta}_{\mathrm{loc}}(B_1)$ be an nonnegative $L^{\theta}$-strong solution of
	\begin{align*}
		\mathcal{M}_{\lambda, \Lambda}^+(D^2u) \geq -|f| \quad \text{and} \quad 	\mathcal{M}_{\lambda, \Lambda}^-(D^2u) \leq |f|  \quad  \text{in $B_1$}.
	\end{align*}
	Then there exists a constant $C>0$ depending only on $\|1/\lambda\|_p$, $\|\Lambda\|_q$, and $n$ such that
	\begin{align*}
		\sup_{B_{1/2}}u \leq C \left(\inf_{B_{1/2}}u+\|f/\lambda\|_{L^n(B_1)}\right).
	\end{align*}
\end{corollary}

\begin{corollary}[A priori H\"older estimates]\label{cor-holder}
	Let $f \in L^{\tau}(B_1)$. Assume that
	\begin{align*}
		\frac{1}{p}+\frac{1}{q}<\frac{1}{2n}.
	\end{align*}
	Moreover, suppose that $u \in W^{2, \theta}_{\mathrm{loc}}(B_1)$ be an $L^{\theta}$-strong solution of
	\begin{align*}
		\mathcal{M}_{\lambda, \Lambda}^+(D^2u) \geq -|f| \quad \text{and} \quad 	\mathcal{M}_{\lambda, \Lambda}^-(D^2u) \leq |f|  \quad  \text{in $B_1$}.
	\end{align*}
	Then there exist constants $C, \alpha>0$ depending only on $\|1/\lambda\|_p$, $\|\Lambda\|_q$, and $n$ such that the following estimate holds:
	\begin{align*}
		\|u\|_{C^{\alpha}(\overline{B_{1/2}})} \leq C\left(\|u\|_{L^{\infty}(B_1)}+\|f/\lambda\|_{L^n(B_1)}\right).
	\end{align*}
\end{corollary}

\begin{proof}[Proof of \Cref{cor-harnack} and \Cref{cor-holder}]
	Once we obtain the local boundedness [\Cref{thm-localboundedness}] and the weak Harnack inequality [\Cref{thm-weakharnack}], both corollaries easily follow; we refer to \cite[Section 4]{CC95} for instance.
\end{proof}

\end{document}